\documentclass[preprint, 12pt, english]{amsart}
\usepackage{amsthm}
\usepackage{amsmath}
\usepackage{latexsym, amssymb}
\usepackage{txfonts}
\usepackage{mathtools}
\usepackage{color}
\usepackage[all]{xy}

\newtheorem{thm}{Theorem}[section] %the resolution could also be [subsection]

\newtheorem{cor}[thm]{Corollary}

\theoremstyle{definition}

\newcommand\operA[2]{{\if!#2!\operatorname{#1}\else{\operatorname{#1}_{#2}^{\phantom{I}}}\fi}} % To be used within Bdefs. Usage: $\operA{N}{K/F}$ produces $N_{K/F}$; $\operA{N}{}$ produces $N$.

\def\tr{{\operatorname{Tr}}}

\newcommand{\Trace}[1][]{\if!#1!\operatorname{Tr}\else{\operatorname{Tr}_{#1}^{\phantom{I}}}\fi} % Usage: $\Tr[K/F](a)$.

\long\def\forget#1\forgotten{{}} %

\def\({\left(}
\def\){\right)}

\newcommand\LAY[3][]{{\begin{array}{c}\mbox{#2} \if#1!{}\else{+}\fi \\ \mbox{#3}\end{array}}}

\makeatletter

\def\ps@pprintTitle{%
 \let\@oddhead\@empty
 \let\@evenhead\@empty
 \def\@oddfoot{}%
 \let\@evenfoot\@oddfoot}

\newcommand{\bigperp}{%
  \mathop{\mathpalette\bigp@rp\relax}%
  \displaylimits
}

\newcommand{\bigp@rp}[2]{%
  \vcenter{
    \m@th\hbox{\scalebox{\ifx#1\displaystyle2.1\else1.5\fi}{$#1\perp$}}
  }%
}
\makeatother

\renewcommand{\geq}{\geqslant}
\renewcommand{\leq}{\leqslant}

\newif\iffurther
\furtherfalse
\begin{document}
%\begin{frontmatter}

\title{Symbol Length of Classes in Milnor $K$-groups}

\author{Adam Chapman}
\address{School of Computer Science, Academic College of Tel-Aviv-Yaffo, Rabenu Yeruham St., P.O.B 8401 Yaffo, 6818211, Israel}
\email{adam1chapman@yahoo.com}

\begin{abstract}
Given a field $F$, a positive integer $m$ and an integer $n\geq 2$, we prove that the symbol length of classes in Milnor's $K$-groups $K_n F/2^m K_n F$ that are equivalent to single symbols under the embedding into $K_n F/2^{m+1} K_n F$ is at most $2^{n-1}$ under the assumption that $F \supseteq \mu_{2^{m+1}}$.
Since for $n=2$, $K_2 F/2^m K_2 F \cong {_{2^m}Br(F)}$, this coincides with the upper bound of $2$ for the symbol length of central simple algebras of exponent $2^m$ that are Brauer equivalent to a single symbol algebra of degree $2^{m+1}$ proved by Tignol in 1983.
We also consider the cases where the embedding into $K_n F/2^{m+1} K_n F$ is of symbol length 2, 3 and 4 (the latter when $n=2$). We finish with studying the symbol length of classes in $K_3/3^m K_3 F$ whose embedding into $K_3 F/3^{m+1} K_3 F$ is one symbol when $F \supseteq \mu_{3^{m+1}}$.
\end{abstract}

%\end{frontmatter}
\maketitle

\section{Introduction}

Given a positive integer $r$ and a field $F$ of $\operatorname{char}(F) \nmid r$ containing a primtiive $r$th root of unity $\rho$, a symbol algebra of degree $r$ over $F$ is an algebra of the form $(\alpha,\beta)_{r,F}=F \langle x,y : x^{r}=\alpha, y^{r}=\beta, yx=\rho xy \rangle$ for some $\alpha,\beta \in F^\times$. 
The group ${_rBr}(F)$ is known (\cite{MS}) to be generated by the Brauer classes of such symbol algebras. The symbol length of an element in ${_rBr(F)}$ is the minimal number of symbol algebras of degree $r$ required to express it.
In \cite{Tignol:1983}, Tignol proved that if a central simple algebra of exponent $r$ is Brauer equivalent to a symbol algebra of degree $t\cdot r$ then its symbol length in ${_rBr}(F)$ is at most $t$.
In particular, taking $r=2^m$ and $t=2$, the upper bound is 2.
We provide here the analogous statement for classes in $K_n F/2^m K_n F$ which are congruent to single symbols when embedded into $K_n F/2^{m+1} K_n F$, showing that the symbol length of those classes in $K_n F/2^m K_n F$ is at most $2^{n-1}$ when $F \supseteq \mu_{2^{m+1}}$.
For the special case of $n=2$ this coincides with Tignol's result, because of the isomorphism $K_2 F/2^m K_2 F \cong {_{2^m}Br}(F)$.
We also consider the cases where the embedding into $K_n F/2^{m+1} K_n F$ is of symbol length  2, 3 and 4 (the latter only when $n=2$), and provide upper bounds for the original classes in $K_n F/2^m K_n F$.
We finish with studying the symbol length of classes in Milnor's $K$-groups $K_3/3^m K_3 F$ whose embedding into $K_3 F/3^{m+1} K_3 F$ is of symbol length 1 when $F \supseteq \mu_{3^{m+1}}$.
In most of the proofs, the trick is to apply a known chain lemma. Such a trick was applied earlier in \cite[Theorem 4.1]{Matzri:2014}.

\section{Preliminaries}

Given a field $F$ and a positive integer $n$, the Milnor $K$-group $K_n F$ is defined as the group of formal sums of symbols $\{\alpha_1,\dots,\alpha_n\}$ with entries in $F^\times$, modulo the relations
\begin{itemize}
\item $\{\alpha_1,\dots,\alpha_{i-1},\alpha_i,\alpha_{i+1}\dots,\alpha_n\}+\{\alpha_1,\dots,\alpha_{i-1},\alpha_i',\alpha_{i+1}\dots,\alpha_n\}\\=\{\alpha_1,\dots,\alpha_{i-1},\alpha_i\alpha_i',\alpha_{i+1}\dots,\alpha_n\}$, and
\item $\{\dots,\alpha,\dots,1-\alpha,\dots\}=0$.
\end{itemize}
The groups we are interested in are $K_n F/\ell K_n F$ in which we have the additional relations $\underbrace{\omega +\dots+\omega}_{\ell \ \text{times}}=0$ for any $\omega\in K_n F$.
This relation gives rise to several identities such as $\{\alpha,-\alpha\}=0$, and so on.
When $\operatorname{char}(F) \nmid \ell$ and $F$ contains the group of $\ell$th roots of unity $\mu_\ell$, $K_n F/\ell K_n F \cong H^n(F,\mu_\ell^{\otimes n})$ by \cite{Voevodsky:2011}.
In the special case of $\ell=2$, $K_n F/2 K_n F \cong I^n F/I^{n+1} F$ by \cite{OVV} and \cite{Voevodsky} (when $\operatorname{char}(F) \neq 2$) and \cite{Kato:1982} (when $\operatorname{char}(F)=2$), where $IF$ is the fundamental ideal of even-dimensional symmetric bilinear forms in the Witt ring $W F$ of $F$.
Recall that 
$$\xymatrix{
1\ar@{->}[r] & K_n F/p^m K_n F\ar@{->}[r]^{\operatorname{Shift}}&K_n F/p^{m+1} K_n F \ar@{->}[r]^{\operatorname{Exp}} &K_n F/ pK_n F \ar@{->}[r] &1
}$$
for any prime integer $p$ and positive integer $m$, where the maps are given by
$$\operatorname{Shift}(\{\alpha_1,\dots,\alpha_n\})=p\cdot \{\alpha_1,\dots,\alpha_n\}, \operatorname{Exp}(\{\alpha_1,\dots,\alpha_n\})=\{\alpha_1,\dots,\alpha_n\},$$
as long as either $\operatorname{char}(F)=p$ (see \cite[Remark 2.32]{AravireJacobORyan:2018}) or $F \subseteq \mu_{p^{m+1}}$ (see \cite[Page 215]{GilleSzamuely:2006}).
In order to avoid confusion, we shall add a subscript indicating the group the symbol lives in: $\{\alpha_1,\dots,\alpha_n\}_{p^m}$ refers to the symbol in $K_n F/p^m K_n F$.
There is also the cup product $\cup$ which maps the pair $$(\{\alpha_1,\dots,\alpha_{n_1}\}_{p^m} , \{\beta_1,\dots,\beta_{n_2}\}_{p^m})\in K_{n_1} F/p^m K_{n_1} F \times K_{n_2} F / p^m K_{n_2} F$$ to $\{\alpha_1,\dots,\alpha_{n_1},\beta_1,\dots,\beta_{n_2}\}_{p^m}\in K_{n_1+n_2}F/p^m K_{n_1+n_2} F.$
The symbol length of a class in $K_n F/p^m K_n F$ is the minimal number of symbols needed to express it as a sum of symbols.
\section{Symbol Length in $K_n F/2^m K_n F$}

\begin{thm}\label{T1}
Let $F$ be a field, $m$ a nonnegative integer and $n$ an integer $\geq 2$.
Suppose that either $\operatorname{char}(F)=2$ or $F\supseteq \mu_{2^{m+1}}$.
Consider a class $\omega \in K_n F/2^m K_n F$ such that $\omega$ is represented by a single symbol when embedded into $K_n F/2^{m+1} K_n F$.
Then its symbol length in $K_n F/2^m K_n F$ is at most $2^{n-1}$.
\end{thm}

\begin{proof}
We prove it by induction on $n$.
Suppose the statement holds true for any number $\leq n-1$.
Let $\omega \in K_n F/2^m K_n F$ and write $\operatorname{Shift}(\omega)=\{\alpha_1,\dots,\alpha_n\}_{2^{m+1}}$ as an element in $K_n F/2^{m+1} K_n F$.
Now, $\operatorname{Exp}(\omega)=\{\alpha_1,\dots,\alpha_n\}_2$ is trivial, and so the corresponding bilinear Pfister form $\varphi_n=\langle \! \langle \alpha_1,\dots,\alpha_n\rangle \! \rangle$ is isotropic.
Write $\varphi_i=\langle \! \langle \alpha_1,\dots,\alpha_i \rangle \! \rangle$ for any $i \in \{1,\dots,n-1\}$.
If $\varphi_{n-1}$ is isotropic, then $\{\alpha_1,\dots,\alpha_{n-1}\}_{2^{m+1}}$ is in the image of $\operatorname{Shift}$,
 and so the symbol length of the $\operatorname{Shift}^{-1}\{\alpha_1,\dots,\alpha_{n-1}\}_{2^{m+1}}$ is at most $2^{n-2}$ by the induction hypothesis, 
 and therefore the symbol length of $\omega=\operatorname{Shift}^{-1}(\{\alpha_1,\dots,\alpha_{n-1}\}_{2^{m+1}}) \cup \{\alpha_n\}_{2^m}$ is at most $2^{n-2}<2^{n-1}$.
Suppose $\varphi_{n-1}$ is anisotropic. Then $-\alpha_n=\varphi_{n-2}(t_{n-2})\alpha_{n-1}+\dots+\varphi_1(t_1)\alpha_2-\varphi_1(t_0)$ where $t_i \in V_{\varphi_i}$ for any $i \in \{1,\dots,n-1\}$ and $t_0\in V_{\varphi_1}$.
For any $i \in \{2,\dots,n-1\}$, let $\beta_i$ be $\varphi_{i-1}(t_{i-1})\alpha_i$ if $t_{i-1}\neq 0$ and $\alpha_i$ otherwise. For convenience, write $\beta_n=\alpha_n$.
By the induction hypothesis, for each $i\in \{2,\dots,n-1\}$, the symbol length of $\operatorname{Shift}^{-1}(\{\alpha_1,\dots,\alpha_{i-1},\beta_i^{-1}\alpha_i\}_{2^{m+1}})$ is at most $2^{i-1}$.
Now, $\{\alpha_1,\dots,\alpha_n\}_{2^{m+1}}$ is the sum
of the symbols 
\begin{eqnarray*}
\{\alpha_1,\beta_2,\dots,\beta_n\}_{2^{m+1}}, \{\alpha_1,\beta_2^{-1} \alpha_2,\beta_3,\dots,\beta_n\}_{2^{m+1}}, \{\alpha_1,\alpha_2,\beta_3^{-1} \alpha_3,\beta_4,\dots,\beta_n\}_{2^{m+1}},\\ \dots,\{\alpha_1,\dots,\alpha_{n-2},\beta_{n-1}^{-1} \alpha_{n-1},\beta_n\}_{2^{m+1}}.
\end{eqnarray*}
Each of these symbols is in the image of $\operatorname{Shift}$, and the symbol lengths of the pre-images of the last $n-1$ symbols are by the induction hypothesis $2,4,\dots,2^{n-2}$.
If $t_0=0$, $\{\beta_2,\dots,\beta_n\}_{2^{m+1}}$ is trivial because $-\beta_n$ is a partial sum of $\beta_2,\dots,\beta_{n-1}$.
If $t_0 \neq 0$, then $\{\alpha_1,\beta_2,\dots,\beta_n\}_{2^{m+1}}$ is congruent to $\{\alpha_1,\dots,\varphi_1(t_0)\}_{2^{m+1}}$, and since the symbol length of $\operatorname{Shift}^{-1}(\{\alpha_1,\varphi_1(t_0)\})$ is at most 2, the symbol length of $\operatorname{Shift}^{-1}(\{\alpha_1,\beta_2,\dots,\beta_n\}_{2^{m+1}})$ is at most 2.
Altogether, the symbol length of $\omega$ is at most $2+2+4+\dots+2^{n-2}=2^{n-1}$.

It remains to explain why the statement is true for $n=2$.
Consider $\operatorname{Shift}(\omega)=\{\alpha,\beta\}_{2^{m+1}}$. Since $\{\alpha,\beta\}_2$ is trivial, $\beta=x^2-\alpha y^2$ for some $x,y \in F$, not both zero.
If $y=0$, then $\omega=\{\alpha,x\}_{2^m}$.
If $x=0$, then $\{\alpha,\beta\}_{2^{m+1}}=\{\alpha,-\alpha y^2\}_{2^{m+1}}=\{\alpha,y^2\}_{2^{m+1}}$, and so $\omega=\{\alpha,y\}_{2^m}$.
Suppose $x,y \neq 0$.
Then, $\{\alpha,x^2-\alpha y^2\}_{2^{m+1}}$ is the sum of $\{\alpha,y^2\}_{2^{m+1}}$ and $\{\alpha,\frac{x^2}{y^2}-\alpha\}_{2^{m+1}}$.
The first symbol is $\operatorname{Shift}(\{\alpha,x\}_{2^m})$. The second symbol is congruent to $\{\frac{x^2}{y^2},-\alpha^{-1}(\frac{x^2}{y^2}-\alpha)\}_{2^{m+1}}$. Therefore $\omega=\{\alpha,x\}_{2^m}+\{\frac{x}{y},-\alpha^{-1}(\frac{x^2}{y^2}-\alpha)\}_{2^m}$.
\end{proof}

As a result of this theorem, we can compute the symbol length of classes whose image under $\operatorname{Shift}$ is of higher symbol length when $\operatorname{char}(F)\neq 2$:
\begin{thm}\label{T2}
Let $F$ be a field of $\operatorname{char}(F)\neq 2$, $m$ a nonnegative integer and $n$ an integer $\geq 2$. Suppose $F\supseteq \mu_{2^{m+1}}$.
Let $\omega$ be a class in $K_n F/2^m K_n F$ such that $\operatorname{Shift}(\omega)=\{\alpha_1,\dots,\alpha_n\}_{2^{m+1}}-\{\beta_1,\dots,\beta_i,\alpha_{i+1},\dots,\alpha_n\}_{2^{m+1}}$. Then the symbol length of $\omega$ is 
at most $i \cdot 2^{n-1}+(i-1)\cdot 2^{n-2}+\dots+1\cdot 2^{n-i}$.
\end{thm}

\begin{proof}
By induction on $i$. The case of $i=1$ is the Theorem \ref{T1}.
Suppose $i\geq 2$ and that the statement holds true for any number $< i$.
The forms $\langle \! \langle \alpha_1,\dots,\alpha_n \rangle \! \rangle$ and $\langle \! \langle \beta_1,\dots,\beta_i,\alpha_{i+1},\dots,\alpha_n\rangle \! \rangle$ are either both hyperbolic or isomorphic.
If they are hyperbolic, then by the Theorem \ref{T1}, the symbol length of $\omega$ is at most $2^{n-1}+2^{n-1}\leq i \cdot 2^{n-1}+(i-1)\cdot 2^{n-2}+\dots+1\cdot 2^{n-i}$.
Suppose they are anisotropic, and therefore isomorphic.
Then $\alpha_i$ is represented by $\langle \! \langle \beta_1,\dots,\beta_i \rangle \! \rangle' \otimes \langle \! \langle \alpha_{i+1},\dots,\alpha_n\rangle \! \rangle$.
Therefore $\alpha_i=\varphi_1(t_1) \beta_1+\dots+\varphi_{i}(t_{i}) \beta_{i}$ where $\varphi_r=\langle \! \langle \beta_1,\dots,\beta_r,\alpha_{i+1},\dots,\alpha_n \rangle \! \rangle$ and $t_r \in V_{\varphi_r}$ for each $r \in \{1,\dots,i\}$.
Let $\gamma_r$ be $\varphi(t_r) \beta_r$ if $t_r \neq 0$ and $\beta_r$ when $t_r=0$, for each $r \in \{1,\dots,i\}$.
The class $\{\beta_1,\dots,\beta_i,\alpha_{i+1},\dots,\alpha_n\}_{2^{m+1}}-\{\alpha_1,\dots,\alpha_n\}_{2^{m+1}}$ is the sum of the classes
\begin{eqnarray*}
&\{\gamma_1,\dots,\gamma_i,\alpha_{i+1},\dots,\alpha_n\}_{2^{m+1}} - \{\alpha_1,\dots,\alpha_n\}_{2^{m+1}}&, \\  &\{\gamma_1,\dots,\gamma_{i-1},\gamma_{i}^{-1} \beta_{i},\alpha_{i+1},\dots,\alpha_n\}_{2^{m+1}}&,\\ 
&\{\gamma_1,\dots,\gamma_{i-2},\gamma_{i-1}^{-1} \beta_{i-1},\beta_i,\alpha_{i+1},\dots,\alpha_n\}_{2^{m+1}}&,\\
&\vdots & \\ & \{\gamma_1^{-1} \beta_1,\beta_2,\dots,\beta_i,\alpha_{i+1},\dots,\alpha_n\}_{2^{m+1}} & .
\end{eqnarray*}
All these classes are in the image of $\operatorname{Shift}$, because their sum is and the last $i$ classes are.
The first class is congruent to $\{\delta_1,\dots,\delta_{i-1},\alpha_i,\dots,\alpha_n\}_{2^{m+1}}$ for some $\delta_1,\dots,\delta_{i-1} \in F^\times$, and so by the induction hypothesis, the pre-image of the first symbol is of symbol length at most $(i-1) \cdot 2^{n-1}+(i-2)\cdot 2^{n-2}+\dots+1\cdot 2^{n-(i-1)}$.
By the Theorem \ref{T1}, the symbol lengths of the pre-images of the other $i$ symbols are $2^{n-1}+2^{n-2}+\dots+2^{n-i}$.
\end{proof}

\begin{cor}
We proved the following: if $\omega$ is a class in $K_n F/2^m K_n F$ such that the symbol length of $\operatorname{Shift}(\omega)$ is at most 2, then the symbol length of $\omega$ is at most $n\cdot  2^{n-1}+(n-1)\cdot 2^{n-2}\dots+1\cdot 1=(n-1)\cdot 2^n+1$.
\end{cor}

\begin{thm}\label{T3}
Suppose $F$ is a field of $\operatorname{char}(F)\neq 2$, $n$ an integer $\geq 2$ and $m$ a positive integer.
Suppose $F\supseteq \mu_{2^{m+1}}$.
Let $\omega$ be a class in $K_n F/2^m K_n F$ such that $\operatorname{Shift}(\omega)$ is of symbol length 3. Then the symbol length of $\omega$ is at most $3\cdot ((n-1)\cdot 2^n+1)$.
\end{thm}

\begin{proof}
Write $\operatorname{Shift}(\omega)=\{\alpha_1,\dots,\alpha_n\}_{2^{m+1}}+\{\beta_1,\dots,\beta_n\}_{2^{m+1}}+\{\gamma_1,\dots,\gamma_n\}_{2^{m+1}}$ as a sum of three symbols.
Then $\langle \! \langle \alpha_1,\dots,\alpha_n\rangle\!\rangle \perp \langle \! \langle \beta_1,\dots,\beta_n\rangle\!\rangle \perp \langle \! \langle \gamma_1,\dots,\gamma_n\rangle\!\rangle$ is trivial mod $I^{n+1} F$, which means by \cite[Corollary 24.6]{EKM} that there exist $a,b,c_2,\dots,c_n \in F^\times$ for which $\langle \! \langle \alpha_1,\dots,\alpha_n\rangle\!\rangle \simeq \langle \! \langle a,c_2,\dots,c_n\rangle\!\rangle$,\\ $\langle \! \langle \beta_1,\dots,\beta_n\rangle\!\rangle \simeq \langle \! \langle b,c_2,\dots,c_n\rangle\!\rangle$ and $\langle \! \langle \gamma_1,\dots,\gamma_n\rangle\!\rangle\simeq \langle \! \langle a^{-1}b^{-1},c_2,\dots,c_n\rangle\!\rangle$.

Now, $\operatorname{Shift}(\omega)-(\{a,c_2,\dots,c_n\}_{2^{m+1}}+\{\beta_1,\dots,\beta_n\}_{2^{m+1}}+\{\gamma_1,\dots,\gamma_n\}_{2^{m+1}})=\{\alpha_1,\dots,\alpha_n\}_{2^{m+1}}-\{a,c_2,\dots,c_n\}_{2^{m+1}}$ is in the image of $\operatorname{Shift}$, and therefore by the Theorem \ref{T2}, the symbol length of its pre-image is at most $n\cdot  2^{n-1}+(n-1)\cdot 2^{n-2}\dots+1\cdot 1$.
Similarly, the classes $\{a,c_2,\dots,c_n\}_{2^{m+1}}+\{\beta_1,\dots,\beta_n\}_{2^{m+1}}+\{\gamma_1,\dots,\gamma_n\}_{2^{m+1}}-(\{a,c_2,\dots,c_n\}_{2^{m+1}}+\{b,c_2,\dots,c_n\}_{2^{m+1}}+\{\gamma_1,\dots,\gamma_n\}_{2^{m+1}})$ and $\{a,c_2,\dots,c_n\}_{2^{m+1}}+\{b,c_2,\dots,c_n\}_{2^{m+1}}+\{\gamma_1,\dots,\gamma_n\}_{2^{m+1}}-(\{a,c_2,\dots,c_n\}_{2^{m+1}}+\{b,c_2,\dots,c_n\}_{2^{m+1}}+\{a^{-1}b^{-1},c_2,\dots,c_n\}_{2^{m+1}})$ are in the image of $\operatorname{Shift}$, and their pre-images are of symbol length at most $(n-1)\cdot 2^n+1$.
Since $\omega$ is the sum of those pre-images, its symbol length is at most $3\cdot((n-1)\cdot 2^n+1)$.
\end{proof}

\begin{thm}\label{T4}
Suppose $F$ is a field of $\operatorname{char}(F)\neq 2$ and $m$ is a positive integer. Suppose $F\supseteq \mu_{2^{m+1}}$. 
If $\omega$ is a class in $K_2 F/2^m K_2 F$ for which the symbol length of $\operatorname{Shift}(\omega)$ is 4, then the symbol length of $\omega$ is at most $46$.
\end{thm}

\begin{proof}
Write $\operatorname{Shift}(\omega)=\{\alpha_1,\beta_1\}_{2^{m+1}} +\{\alpha_2,\beta_2\}_{2^{m+1}}-(\{\alpha_3,\beta_3\}_{2^{m+1}}+\{\alpha_4,\beta_4\}_{2^{m+1}})$.
Since $\{\alpha_1,\beta_1\}_2 +\{\alpha_2,\beta_2\}_2-(\{\alpha_3,\beta_3\}_2+\{\alpha_4,\beta_4\}_2)=0$, we have $\{\alpha_1,\beta_1\}_2+\{\alpha_2,\beta_2\}_2=\{\alpha_3,\beta_3\}_2+\{\alpha_4,\beta_4\}_2$.
By \cite{Sivatski:2012}, there exist $$a_1,b_1,c_1,d_1,a_2,b_2,c_2,d_2,a_3,b_3,c_3,d_3 \in F^\times , \ \text{such that}$$
\begin{itemize}
\item  $\{\alpha_1,\beta_1\}_2=\{a_1,b_1\}_2$, $\{\alpha_2,\beta_2\}_2=\{c_1,d_1\}_2$.
\item For some $t_1 \in V_{\varphi_1}$, $\{a_1,b_1 \varphi_1(t_1)\}_2=\{a_2,b_2\}_2$ and $\{c_1,d_1\varphi_1(t_1)\}_2=\{c_2,d_2\}_2$ where $\varphi_1=\langle \! \langle a_1c_1\rangle \! \rangle$.
\item For some $t_2 \in V_{\varphi_2}$, $\{a_2,b_2 \varphi_2(t_2)\}_2=\{a_3,b_3\}_2$ and $\{c_2,d_2\varphi_2(t_2)\}_2=\{c_3,d_3\}_2$ where $\varphi_2=\langle \! \langle a_2c_2\rangle \! \rangle$.
\item For some $t_3 \in V_{\varphi_3}$, $\{a_3,b_3 \varphi_3(t_3)\}_2=\{\alpha_3,\beta_3\}_2$ and $\{c_3,d_3\varphi_3(t_3)\}_2=\{\alpha_4,\beta_4\}_2$ where $\varphi_3=\langle \! \langle a_3c_3\rangle \! \rangle$.
\end{itemize}
The class $\{\alpha_1,\beta_1\}_{2^{m+1}}-\{a_1,b_1\}_{2^{m+1}}$ is in the image of $\operatorname{Shift}$, and its pre-image is of symbol length at most 5. The same goes for the classes 
\begin{eqnarray*}
\{a_1,b_1\varphi_1(t_1)\}_{2^{m+1}}-\{a_2,b_2\}_{2^{m+1}},\\
\{a_2,b_2 \varphi_2(t_2)\}_{2^{m+1}}-\{a_3,b_3\}_{2^{m+1}},\\ 
\{a_3,b_3 \varphi_3(t_3)\}_{2^{m+1}}-\{\alpha_3,\beta_3\}_{2^{m+1}},\\ 
\{\alpha_2,\beta_2\}_{2^{m+1}}-\{c_1,d_1\}_{2^{m+1}},\\ 
\{c_1,d_1\varphi_1(t_1)\}_{2^{m+1}}-\{c_2,d_2\}_{2^{m+1}},\\ 
\{c_2,d_2 \varphi_2(t_2)\}_{2^{m+1}}-\{c_3,d_3\}_{2^{m+1}},\\
\{c_3,d_3 \varphi_3(t_3)\}_{2^{m+1}}-\{\alpha_4,\beta_4\}_{2^{m+1}}.
\end{eqnarray*}
The class $\{a_1,b_1\}_{2^{m+1}}+\{c_1,d_1\}_{2^{m+1}}-(\{a_1,b_1\varphi_1(t_1)\}_{2^{m+1}}+\{c_1,d_1\varphi_1(t_1)\}_{2^{m+1}})$ is congruent to $\{a_1c_1,\varphi_1(t_1)^{-1}\}_{2^{m+1}}$. Therefore, it is in the image of $\operatorname{Shift}$, and its pre-image is of symbol length at most 2.
The same applies to the classes $\{a_2,b_2\}_{2^{m+1}}+\{c_2,d_2\}_{2^{m+1}}-(\{a_2,b_2\varphi_2(t_2)\}_{2^{m+1}}+\{c_2,d_2\varphi_2(t_2)\}_{2^{m+1}})$ and $\{a_3,b_3\}_{2^{m+1}}+\{c_3,d_3\}_{2^{m+1}}-(\{a_3,b_3\varphi_3(t_3)\}_{2^{m+1}}+\{c_3,d_3\varphi_3(t_3)\}_{2^{m+1}})$.
Since $\operatorname{Shift}(\omega)$ is the sum of the eleven classes described above, $\omega$ is of symbol length at most $8\cdot 5+3\cdot 2=46$.
\end{proof}

\section{Symbol Length in $K_n F/3^m K_n F$}

In this last section, we consider the situation for the Milnor $K$-groups mod 3.
Suppose $F$ is a field of $\operatorname{char}(F)\neq 3$ containing $\mu_{3^{m+1}}$.
Then by \cite{MS}, $K_2 F/3 K_2 F \cong {_3Br}(F)$, and therefore by \cite{Tignol:1983}, every class $\omega \in K_2 F/3^m K_2 F$ for which $\operatorname{Shift}(\omega)$ is of symbol length 1, is of symbol length at most 3.
We shall now consider the analogous situation in $K_3 F/3^m K_n F$.
For any class $\omega \in K_n F/3^m K_n F$ and field extension $E/F$, $\omega_E$ is the image of $\omega$ in $K_n E/3^m K_nE$.
\begin{thm}\label{T5}
Suppose $F$ is a field of $\operatorname{char}(F)\neq 3$ and $F \supseteq \mu_{3^{m+1}}$.
Let $\omega \in K_3 F/3^m K_3 F$.
If $\operatorname{Shift}(\omega)$ is of symbol length 1, then the symbol length of $\omega$ is at most $30$.
If $F$ is quadratically closed, then the symbol length of $\omega$ is at most 15.
\end{thm}

\begin{proof}
Write $\operatorname{Shift}(\omega)=\{\alpha,\beta\gamma\}_{3^{m+1}}$.
By \cite{Suslin:1985}, since $\{\alpha,\beta,\gamma\}_3$ is trivial, $\gamma$ is the reduced norm $\operatorname{Norm}(z)$ of some element $z$ in the symbol algebra $A=(\alpha,\beta)_{3,F}$.
Write $E$ for the maximal subfield of $A$ generated by $x=\sqrt[3]{\alpha}$.
Consider the equations $\tr((a+bx+cx^2)z)=0$ and $\tr(((a+bx+cx^2)z)^2)=0$ where $a,b,c \in F$ where $\tr$ is the reduced trace map.
There exists a nontrivial solution $(a_0,b_0,c_0)$ to these equations in some field $E$ which is either $F$ or a quadratic extension of $F$. In particular $E=F$ when $F$ is quadratically closed.
Write $w=(a_0+b_0x+c_0x^2)z$ in $A \otimes E$.
Then $w^3=\delta \in F^\times$, where $\frac{\delta}{\gamma}=\operatorname{Norm}_{E[x]/E}(a_0+b_0 x+c_0 x^2)$.
By \cite{Rost:1999}, there exist $c_1,c_2,c_3 \in F^\times$ for which the following symbols are isomorphic in $K_2 E/3 K_2 E$:
\begin{itemize}
\item $\{\alpha,\beta\}_3$.
\item $\{\alpha,c_1\}_3$.
\item $\{c_2,c_1\}_3$.
\item $\{c_2,c_3\}_3$.
\item $\{\delta,c_3\}_3$.
\end{itemize}
Now, in $K_3 E/3^{m+1} K_3 E$, $\{\alpha,\beta,\gamma\}_{3^{m+1}}$ is the sum of the following classes
\begin{itemize}
\item $\{\alpha,\beta,\gamma\}_{3^{m+1}}-\{\alpha,\beta,\delta\}_{3^{m+1}}$.
\item $\{\alpha,\beta,\delta\}_{3^{m+1}}-\{\alpha,c_1,\delta\}_{3^{m+1}}$.
\item $\{\alpha,c_1,\delta\}_{3^{m+1}}-\{c_2,c_1,\delta\}_{3^{m+1}}$.
\item $\{c_2,c_1,\delta\}_{3^{m+1}}-\{c_2,c_3,\delta\}_{3^{m+1}}$.
\item $\{c_2,c_3,\delta\}_{3^{m+1}}-\{\delta,c_3,\delta\}_{3^{m+1}}$.
\end{itemize}
Each of the classes above is in the image of $\operatorname{Shift}$, and its pre-image is of symbol length at most $3$ in $K_3 E/3^m K_3 E$, and so the symbol length of $\omega_E$ in $K_3 E/3^m K_3 F$ is at most $15$.
If $E=F$ then we are done.
If $E$ is a quadratic extension of $F$, then since $\operatorname{cor}_{E/F}(\operatorname{res}_{E/F}(\omega))=4 \omega$ is of the same symbol length as $\omega$, the symbol length of $\omega$ in $K_3 F/3^m K_3 F$ is at most $30$.
\end{proof}

\section*{Acknowledgements}
The author thanks Jean-Pierre Tignol for his helpful comments on the manuscript.

\bibliographystyle{alpha}
\def\cprime{$'$}

\end{document}